\documentclass[a4paper, reqno]{amsart}

\usepackage{amsfonts, amsmath, amssymb, enumerate, esint}  
\usepackage{appendix}
\usepackage{tabularx}

\usepackage{xcolor}

\usepackage{hyperref}

\newcommand{\R}{\mathbb{R}}

\newcommand{\Z}{\mathbb{Z}}
\newcommand{\N}{\mathbb{N}}

\theoremstyle{plain}
\newtheorem{theorem}{Theorem}[section]
\newtheorem{lemma}[theorem]{Lemma}
\newtheorem{proposition}[theorem]{Proposition}

\theoremstyle{definition}
\newtheorem{definition}[theorem]{Definition}

\newcommand{\rt}[1]{\textrm{#1}}

\newcommand{\schw}[1]{\mathcal{S}(\mathbb{R}^{#1})}
\newcommand{\modelcurves}[1]{\mathfrak{S}(#1)}
\newcommand{\supp}[1]{\mathrm{supp} \ #1}
\newcommand{\spsupp}[2]{\mathrm{supp}_{#1} \ #2}
\newcommand{\cexp}{2}

\newcommand{\vecr}{\delta}
\newcommand{\vecrr}{\mathbf{r}}
\newcommand{\ar}{a_{\vecr}}
\newcommand{\arr}{a_{\vecrr}}
\newcommand{\ezero}{\varepsilon_{0}}
\newcommand{\eone}{\varepsilon_{1}}

\newcommand{\dxi}{\mathrm{d}\xi}

\newcommand{\dy}{\mathrm{d}y}

\newcommand{\dt}{\mathrm{d}t}
\newcommand{\ds}{\mathrm{d}s}

\newcommand{\ft}[2]{\mathcal{F}_{#1}(#2)}
\newcommand{\ift}[2]{\mathcal{F}_{#1}^{-1}(#2)}

\newcommand{\firstavg}{\mathcal{A}_{\vecr}^{\gamma}}
\newcommand{\maxfn}{\mathcal{N}_{\vecr}^{\gamma}}

\newcommand{\avg}[2]{\mathcal{A}[#1,#2]}
\newcommand{\tavg}[2]{\tilde{\mathcal{A}}[#1,#2]}

\newcommand{\avga}[2]{\mathcal{A}_{\textrm{main}}[#1,#2]}
\newcommand{\avgb}[2]{\mathcal{A}_{\textrm{err}}[#1,#2]}

\newcommand{\fravg}[2]{\mathfrak{D}_s\mathcal{A}[#1,#2]}

\newcommand{\inn}[2]{\langle #1, #2 \rangle}
\newcommand{\norm}[1]{\left \lVert #1 \right \rVert}

\title[A Nikodym maximal function associated to space curves]{$L^2$ estimates for a Nikodym maximal function associated to space curves}
\author{Aswin Govindan Sheri}
\address{School of Mathematics, James Clerk Maxwell Building, The King's Buildings, Peter Guthrie Tait Road, Edinburgh, EH9 3FD, UK.}
\email{a.govindan-sheri@sms.ed.ac.uk}
\date{\today}

\begin{document}

\maketitle


\begin{abstract}
We consider the $L^p \rightarrow L^p$ boundedness of a Nikodym maximal function associated to a one-parameter family of tubes in $\R^{d+1}$ whose directions are determined by a non-degenerate curve $\gamma$ in $\R^d$. These operators arise in the analysis of maximal averages over space curves. The main theorem generalises the known results for $d = 2$ and $d = 3$ to general dimensions. The key ingredient is an induction scheme motivated by recent work of Ko--Lee--Oh.
\end{abstract}


\section{Introduction}\label{sec: intro}
Consider a $C^{\infty}$ non-degenerate curve $\gamma : I : = [-1,1] \rightarrow \R^d$. In other words, 
\begin{align*}
    \det \begin{pmatrix}
       \gamma^{(1)}(s) & \cdots & \gamma^{(d)}(s) 
   \end{pmatrix} \neq 0 \quad \textrm{for all $s \in I$}.
\end{align*}
The curve $\gamma$ defines a one-parameter family of directions in $\R^{d+1}$. For $ 0 < \delta < 1$ and $s \in I$, consider a $\delta$-tube in $\R^{d+1}$ in the direction of $\gamma(s)$, defined as 
\[T_{\delta}(s) := \{(y,t) \in \R^{d} \times I : |y - t\gamma(s)| \leq \delta \}\]
and the corresponding averaging operator
\begin{align}\label{eq: averaging op}
    \firstavg g(x,s) & : = \frac{1}{|T_{\vecr}(s)|}\int_{{T}_{\vecr}(s)} g(x - y,t) \dy \dt, \qquad \text{ for } x \in \R^d
\end{align} 
whenever $g \in L^1_{\mathrm{loc}}(\R^{d+1})$. Our goal is to investigate the $L^p$ boundedness properties of the Nikodym maximal function
\begin{align}\label{eq: main eq.}
 \maxfn g(x) & : = \sup_{s \in I} |\firstavg g(x,s)|.
\end{align}
The main result is as follows.
\begin{theorem}\label{thm: main} Let $\gamma: I \rightarrow \R^d$ be a non-degenerate curve. There exists $C_{d,\gamma} > 0$ such that
    \begin{align*} 
\norm{\maxfn}_{L^2(\R^{d+1}) \rightarrow L^2(\R^{d})} \leq C_{d,\gamma} (\log \delta^{-1})^{d/2} \qquad \textrm{ for all $0 < \delta < 1$}.
    \end{align*}
\end{theorem}

By interpolating with the trivial bound at $L^{\infty}$, we estimate the $L^p$ operator norm for the maximal function as $O((\log \delta^{-1})^{{d}/{p}})$ for $2 \leq p \leq \infty$. This is sharp in the sense that the $L^p$ operator norm has polynomial blowup in $\delta^{-1}$ for $1 \leq p < 2$ (see \S\ref{sec: sharpness of main thm}). The result is new for $d \geq 4$. The theorem also slightly strengthens the known estimates for $d = 2$ and $d = 3$ (see \cite[Lemma 1.4]{mockseegsogg} and \cite[Proposition 5.5]{bghs-lp}, respectively) by improving the dependence on $\delta^{-1}$.

The operator $\maxfn$ is a variant of the classical Nikodym maximal function considered in \cite{Cordoba_forward_squarefn}. The main difference lies in the dimensional setup of the problem: by the above definition, $\maxfn$ maps functions on $\R^{d+1}$ to functions on $\R^d$, whereas the classical operator considered in \cite{Cordoba_forward_squarefn} is a mapping between functions on the same Euclidean space. 

Maximal functions of the form \eqref{eq: main eq.} naturally arise in the study of local smoothing problems for averaging operators associated to $\gamma$, as first observed in Mockenhaupt--Seeger--Sogge \cite{mockseegsogg}. In \cite[Lemma~1.4]{mockseegsogg} estimates for $\maxfn$ were obtained for $d = 2$. The $d=3$ case was later considered in \cite[Propostion~5.5]{bghs-lp}, in relation to the problem of bounding the helical maximal function. The averages $\firstavg$ are also closely related to the restricted X-ray transforms considered in \cite{Pramanik_Seeger(Xray),ko--Lee--Oh_X_ray_transform,Greenleaf_X_ray_L^2}.

Our method differs from the existing literature \cite{mockseegsogg, bghs-lp} in two key respects. First, we use a fractional Sobolev embedding argument to dominate the maximal function by a Fourier integral operator (see Proposition~\ref{prop: Sobolev embedding}). This allows us to fully access orthogonality in the subsequent decomposition. Secondly, we use an induction scheme, which hides the complexity of the root analysis in \cite{bghs-lp}. The induction is motivated by \cite{Ko_Lee_Oh(Local_smoothing)}, where a (more complex) induction argument is used to investigate the local smoothing problem associated to averages along curves in $\R^d$.


\subsection*{Outline of the paper}
This paper is structured as follows: 
\begin{itemize}
    \item In \S\ref{sec: initial reduct} we reduce the proof of Theorem~\ref{thm: main} to Proposition~\ref{prop: main} via fractional Sobolev embedding.
    \item In \S\ref{sec: proof of Main prop} we present the inductive proof of Proposition~\ref{prop: main}.
    \item In \S\ref{sec: sharpness of main thm} we discuss the sharpness of Theorem~\ref{thm: main}.
    \item In \S\ref{sec: further ext} we state an anisotropic extension of Theorem~\ref{thm: main} and briefly discuss its proof.
\end{itemize}

\subsection*{Acknowledgment}
The author is indebted to Jonathan Hickman for both suggesting the problem and many helpful suggestions in the development of this paper.

The author was supported by The Maxwell Institute Graduate School in Analysis and its Applications, a Centre for Doctoral Training funded by the UK Engineering and Physical Sciences Research Council (grant EP/L016508/01), the Scottish Funding Council, Heriot--Watt University, and the University of Edinburgh.


\section{Notational conventions} \label{sec: notation}
For a set $E \subseteq \R^n$, we denote its characteristic function by $\chi_{E}$. Given $f \in L^1(\R^n)$ we let either $\hat{f}$ or $\mathcal{F}(f)$ denote its Fourier transform and $\check{f}$ or $\mathcal{F}^{-1}(f)$ denote its inverse Fourier transform, which are normalised as follows:
\begin{equation*}
    \hat{f}(\xi) := \int_{\R^n} e^{-i x \cdot \xi} f(x)\,\mathrm{d} x, \qquad \check{f}(\xi) := \int_{\R^n} e^{i x \cdot \xi} f(x)\,\mathrm{d} x.
\end{equation*}

For $m \in L^{\infty}(\R^n) $, we denote by $m(\tfrac{1}{i}\partial_{x})$ the Fourier multiplier operator defined by its action on $g \in \schw{}$ as
\[\ft{}{m(\tfrac{1}{i}\partial_{x})g}(\xi) := m(\xi)\ft{}{g}(\xi) \qquad \textrm{ for $\xi \in \R^n$}. \]

Finally, given two non-negative real numbers $A,B$, and a list of parameters $M_1, \dots, M_n$, the notation $A \lesssim_{M_1, \dots, M_n} B$ or $A = O_{M_1, \dots, M_n}(B)$ signifies that $A \leq C B$ for some constant $C = C_{M_1, \dots, M_n} > 0$ depending only on the parameters $M_1, \dots, M_n$. In addition, $A \sim_{M_1, \dots, M_n} B$ is used to signify that both $A \lesssim_{M_1, \dots, M_n} B$ and $B \lesssim_{M_1, \dots, M_n} A$ hold.


\section{Initial reductions and Sobolev embedding}\label{sec: initial reduct}
\subsection{Initial reductions} Let $I := [-1,1]$ and $\gamma \colon I \to \R^d$ be a non-degenerate curve, as in \S\ref{sec: intro}. We begin by replacing the classical averaging operators by Fourier integral operators. Given ${a \in L^{\infty}( \R^{d} \times I \times I)}$, consider 
\begin{align}\label{eq: defining fio avg}
    \avg{a}{\gamma} g(x,s) & : = \int_{I}\int_{\R^d} e^{i \inn{x - t\gamma(s)}{\xi}} a(\xi,s,t)\ft{x}{g}(\xi,t) \dxi \dt  \qquad \textrm{for $g \in \schw{d+1}$,}
\end{align}
where $\ft{x}{g}(\xi,t)$ denotes $\ft{x}{g(\,\cdot\,,t)}(\xi)$, the Fourier transform of $g$ in $x$ only.
Define the associated maximal operator
\begin{align*}
    \mathcal{N}[a,{\gamma}]g(x,s) & : = \sup_{s \in I} |\avg{a}{\gamma}g(x,s)|.
\end{align*} 

Choose a function $\psi \in C_{c}^{\infty}(\R)$ with $\supp{\psi} \subseteq [-1,1]$ such that its inverse Fourier transform $\check{\psi}$ is non-negative and $\check{\psi}(y) \gtrsim 1$ whenever $|y| \leq 1$. Let $\tilde{\chi}_{I}$ be a non-negative smooth function that satisfies $\tilde{\chi}_{I}(x) = 1$ for all $x \in I$ and $\tilde{\chi}_{I}(x) = 0$ when $x \notin [-2,2]$. Define
\begin{align}\label{eq: symb_defn}
 \ar (\xi,s,t) := \psi(\delta|\xi|) \tilde{\chi}_{I}(s)\tilde{\chi}_{I}(t). 
\end{align}

Let $K_{\vecr}$ denote the kernel of the averaging operator $\firstavg$ defined in \eqref{eq: averaging op}. In particular,
\begin{equation*}
    K_{\vecr}(x,s,t) := \frac{1}{|T_{\vecr}(s)|} \chi_{T_{\vecr}(s)}(x,t). 
\end{equation*}
By integral formula for the inverse Fourier transform and a change of variable,
\[ K_{\vecr}(x,s,t) \lesssim_{d} \int_{\R^d} e^{i\inn{x - t\gamma(s)}{\xi}} \ar(\xi,s,t) \dxi. \]
Thus, the pointwise estimate \[ |\firstavg g (x,s)| \lesssim_{d} |\mathcal{A}[\ar,\gamma]g(x,s)| \] holds. It is therefore enough to bound the operator $\mathcal{N}[\ar,\gamma]$. 

We now perform an endpoint Sobolev embedding to replace the $L^{\infty}_s$ norm in the maximal function with an $L^2_{s}$ norm. Here we write 
\[ \mathfrak{D}_{s}\avg{a}{\gamma} := (1+ \sqrt{-\partial_{s}^{2}})^{1/2} \avg{a}{\gamma}, \] 
where $a$ and $\gamma$ are as above and $(1+ \sqrt{-\partial_{s}^{2}})^{1/2}$ is the fractional differential operator in $s$ with multiplier $(1+|\sigma|)^{1/2}$. 

\begin{proposition}\label{prop: Sobolev embedding} For a nondegenerate curve $\gamma: I \rightarrow \R^d$ , $0< \delta < 1$ and $\ar$ as defined in \eqref{eq: symb_defn}, we have
\[    \norm{\mathcal{N}[\ar,\gamma]g}_{L^2(\R^{d})} \lesssim_{} |\log \delta|^{1/2}\norm{\fravg{\ar}{\gamma}g}_{L^2(\R^{d+1})} + \norm{g}_{L^2(\R^{d+1})}\]
for all $g \in \schw{d+1}$.
\end{proposition}
\begin{proof}
Let $\tilde{\chi}: \R \rightarrow [0,1]$ satisfy $\tilde{\chi}(\sigma) = 1$ for all $\sigma \in (-C\delta^{-1}, C\delta^{-1})$ and $\tilde{\chi}(\sigma) = 0$ when $\sigma \notin (-2C\delta^{-1},2C\delta^{-1})$. 
The constant $C$ is chosen large enough to satisfy the requirements of the forthcoming argument. Defining
\begin{align*}
    \avga{\ar}{\gamma}  := \tilde{\chi}(\tfrac{1}{i} \partial_s) \circ \avg{\ar}{\gamma} \qquad \textrm{and} \qquad \avgb{\ar}{\gamma}  := \avg{\ar}{\gamma} - \avga{\ar}{\gamma},
\end{align*}
where the multiplier operator $\tilde{\chi}(\tfrac{1}{i} \partial_s)$ is defined in \S\ref{sec: notation},
it suffices to prove 
\begin{align}
\norm{\avga{\ar}{\gamma}g}_{L^{2}_{x}L^{\infty}_{s}(\R^{d} \times I)} & \lesssim |\log \delta|^{1/2}\norm{\fravg{a_\vecr}{\gamma}g}_{L^2(\R^{d+1} )} \label{eq: umain estimates},\\
\norm{\avgb{\ar}{\gamma}}_{L^{2}_{x}L^{\infty}_{s}(\R^{d} \times I)} & \lesssim \norm{g}_{L^2(\R^{d+1})} \label{uerror estimates}
\end{align}
for all $g \in \schw{d+1}$.

To prove \eqref{eq: umain estimates}, fix $g \in \schw{d+1}$ and write
\begin{align*}
     \avga{\ar}{\gamma} g(x,s)
     = \tilde{\chi}_1(\tfrac{1}{i} \partial_s) \circ \fravg{\ar}{\gamma}g(x,s) \qquad \textrm{for  $(x,s) \in \R^d \times I,$}
\end{align*}
where $\tilde{\chi}_1(\sigma) := (1+|\sigma|)^{-1/2}\tilde{\chi}(\sigma)$.
Temporarily fix $x \in \R^d$. The above expression can be written as a convolution product in $s$ variable between $\ift{s}{\tilde{\chi}_1}$ and $\fravg{\ar}{\gamma}g(x, \, \cdot \,).$ Using Young's inequality, Plancherel's theorem and by noting that the $L^2$ norm of $\tilde{\chi}_1$ is $O(|\log \delta|^{1/2})$, we obtain 
\begin{align*}
   \norm{\avga{\ar}{\gamma} g(x,\,\cdot\,)}_{L^{\infty}_s(I)} \lesssim_{} |\log \delta|^{1/2}\norm{\fravg{\ar}{\gamma}g(x,\,\cdot\,)}_{L^2_s(\R)}.
\end{align*}
Combining Fubini's theorem with the above estimate for each $x \in \R^d$, we obtain \eqref{eq: umain estimates}.

To prove \eqref{uerror estimates}, write
\begin{align*}
    \avgb{\ar}{\gamma} g & = (1+ \sqrt{-\partial^{2}_{s}})^{-1} \circ (1+ \sqrt{-\partial^{2}_{s}}) \circ (1 - \tilde{\chi})(\tfrac{1}{i}\partial_{s})\circ \avg{\ar}{\gamma}g  \\
    & = \tilde{\chi}_2(\tfrac{1}{i} \partial_s) \circ \tilde{\chi}_3(\tfrac{1}{i} \partial_s) \circ \avg{\ar}{\gamma}g
\end{align*}
where  
\[\tilde{\chi}_{2}(\sigma) := (1+ |\sigma|)^{-1}(1-\tilde{\chi}(\sigma))^{1/2} \quad \textrm{and} \quad  \tilde{\chi}_{3}(\sigma) := (1+ |\sigma|)(1-\tilde{\chi}(\sigma))^{1/2} \]
for $\sigma \in \R$. Note that $(1+ |\sigma|)^{-1}(1-\tilde{\chi}(\sigma))^{1/2}$ has uniformly bounded $L^2$ norm (in $\delta$). Thus, an application of Young's convolution inequality gives
\[ \norm{\avgb{\ar}{\gamma} g(x,\cdot)}_{L^{\infty}_{s}(I)} \lesssim_{} \norm{\tilde{\chi}_3(\tfrac{1}{i} \partial_s) \circ \avg{\ar}{\gamma}g(x,\cdot)}_{L^2_s(\R)} \qquad \textrm{ for } x \in \R^d.\]
Integrating in  $x$ using Fubini's theorem, 
\[  \norm{\avgb{\ar}{\gamma} g}_{L^{2}_{x}L^{\infty}_{s}(\R^{d} \times I)} \lesssim_{} \norm{\tilde{\chi}_3(\tfrac{1}{i} \partial_s) \circ \avg{\ar}{\gamma}g}_{L^2(\R^{d+1})}. \]
By Plancherel's theorem, the quantity on the right can be estimated from above by $L^2$ norm of the function $\mathcal{B}_{\textrm{err}}[\ar, \gamma,  \tilde{\chi_{3}}]g$, where 
\begin{align*}
      \mathcal{B}_{\textrm{err}}[\ar, \gamma,  \tilde{\chi_{3}}]g(\xi,\sigma) & := \int_{I} b_{\vecr}(\xi, \sigma, t)  \ft{x}{g}(\xi,t) \dt
\end{align*}
for
\begin{align}\label{eq: br--oscillatory}
 b_{\vecr}(\xi,\sigma, t) & :=  \tilde{\chi}_3(\sigma) \int_{I}  e^{-i( \sigma s + t\inn{\gamma(s)}{\xi})} \ar(\xi,s,t)  \ds.
\end{align} 
By Minkowski's integral inequality and Plancherel's theorem,
\begin{align*}
    \norm{\mathcal{B}_{\textrm{err}}[\ar, \gamma,  \tilde{\chi_{3}}]g}_{L^2(\R^{d+1})} \lesssim  \norm{b_{\vecr}}_{L^{\infty}_{\xi,t}L^2_{\sigma}(\R^{d} \times I \times \R)} \norm{{g}}_{L^2(\R^{d} \times I)}.
\end{align*}
Thus, the proof of \eqref{uerror estimates} boils down to the estimate $\norm{b_{\vecr}(\xi, \, \cdot \, , t)}_{{L^2_{\sigma}(\R)}} \lesssim_{} 1$ uniformly in $(\xi,t) \in \R^{d} \times I$. Since $|\xi| \lesssim \delta^{-1}$ and $C$ is large, 
\[ |\sigma + t \inn{\gamma'(s)}{\xi}| \sim |\sigma| \quad \textrm{ whenever} \quad (\xi,s,t) \in \supp{\ar},\ \sigma \in\ \supp{\tilde{\chi}_3}. \]
Noting the easy estimate $|\partial_{s}^{\beta}\ar(\xi,s,t)| \lesssim_{\beta} 1$ for all $\beta \in \N$, we apply integration-by-parts to estimate the oscillatory integral in \eqref{eq: br--oscillatory}. In particular,
\[ b_{\vecr}(\xi,\sigma, t) = O_{N,\gamma}((1+ |\sigma|)^{-N}) \qquad \textrm{$(\xi,t) \in \R^{d} \times I$ and $N \geq 1$}. \] It is evident that the required $L^2$ estimate for $b_{\vecr}$ follows from this rapid decay, completing the proof of \eqref{uerror estimates}.
\end{proof}
Proposition~\ref{prop: Sobolev embedding} reduces the analysis to estimating the operator $\fravg{\ar}{\gamma}$. We begin by dyadically decomposing the frequency space. Suppose $\eta, \beta \in C_{c}^{\infty}(\R)$ are the classical Littlewood--Paley functions such that 
\begin{align}\label{eq: littlewood_paley}
\supp{\eta} \subseteq \{ r \in \R: |r| \leq 2 \}, \qquad \supp{\beta} \subseteq \{ r \in \R: {1}/{2} \leq |r| \leq 2 \}    
\end{align}
and 
\begin{align*}
    \eta(r) + \sum_{\lambda \in 2^{\N}} \beta(r/\lambda) = 1 \qquad \textrm{for all $ r \in \R$.}
\end{align*} 
For $\lambda \in \{0\} \cup 2^{\N}$, introduce the dyadic symbols 
\[ a^{\lambda}_{\vecr}({\xi,s,t}) :=  \begin{cases}
    \ar(\xi,s,t)\eta(|\xi|) &\text{if} \ \lambda = 0, \\
\ar(\xi,s,t)\beta(|\xi|/\lambda) & \text{if} \ \lambda \in 2^{\N}.
\end{cases} \]

Theorem~\ref{thm: main} is a consequence of the following result.
\begin{proposition}\label{prop: main}
 Let $\lambda \in \{0\} \cup 2^{\N}$ and $0 < \delta < 1$. Then, 
\begin{align*}
\norm{\fravg{a^{\lambda}_{\vecr}}{\gamma}}_{L^2(\R^{d+1}) \rightarrow L^2(\R^{d+1})} \lesssim_{d,\gamma} (\log (2 + \lambda) )^{(d-1)/2}.
    \end{align*}    
\end{proposition}
\begin{proof}[Proposition~\ref{prop: main} $\implies$ Theorem~\ref{thm: main}] Let $\tilde{\eta}, \tilde{\beta} \in C_c^{\infty}(\R)$ be two non-negative functions such that $\tilde{\eta}(r) = 1$ for $r \in \supp{\eta}$, $ \tilde{\beta}(r) = 1$ for $r \in \supp{\beta}$ and
\[
\tilde{\eta}(r) + \sum_{\lambda \in 2^{\N}} \tilde{\beta}(r/\lambda)\lesssim 1 \qquad \textrm{for all $ r \in \R$.}\]
For $g \in \schw{d+1}$, define
\begin{align*}
    g^{\lambda} := \begin{cases}
        \tilde{\eta}\left(|\tfrac{1}{i}\partial_{x}|\right)g & \textrm{ if $\lambda = 0$},\\[2pt]
        \tilde{\beta}\left(|\tfrac{1}{i}\partial_{x}/{\lambda}|\right)g   & \textrm{ if $\lambda \geq 1$}.
    \end{cases}
\end{align*}
It is clear from the definitions that $\fravg{\ar^{\lambda}}{\gamma}g  = \fravg{a^{\lambda}_{\vecr}}{\gamma}g^{\lambda}$.
By Plancherel's theorem and the support properties of the $a^{\lambda}_{\vecr}$, we have 
\begin{align*}
    \norm{\fravg{\ar}{\gamma}g }_{L^2(\R^{d+1})}^2 \lesssim_{d} \sum_{\lambda \in \{0\} \cup 2^{\mathbb{N}}} \norm{\fravg{a^{\lambda}_{\vecr}}{\gamma}g^{\lambda} }^2_{L^2(\R^{d+1})}.
\end{align*}
Applying Proposition~\ref{prop: main} for each $\lambda$ and observing that $\ar^{\lambda} = 0$ when $\delta^{-1} \lesssim_{d} \lambda$, we obtain
\begin{align*}
    \norm{\fravg{\ar}{\gamma}g }^2_{L^2(\R^{d+1})} & \lesssim_{d,\gamma} (\log (2+ \lambda))^{d-1}\sum_{\lambda \in \{0\} \cup 2^{\mathbb{N}} } \norm{g^{\lambda}}^2_{L^2(\R^{d+1})} \\
    & \lesssim_{d,\gamma} (\log \delta^{-1})^{d-1} \norm{g}^2_{L^2(\R^{d+1})}. 
\end{align*}
Combining the above inequality with Proposition~\ref{prop: Sobolev embedding}, we deduce Theorem~\ref{thm: main}.
\end{proof}
The multiplier associated to  $\fravg{a^{0}_{\vecr}}{\gamma}$ is a bounded function and so the $\lambda = 0$ case of Proposition~\ref{prop: main} is immediate. More interesting cases arise when $\lambda \in 2^{\N}$.


\section{The proof of Proposition~\ref{prop: main}}\label{sec: proof of Main prop}

\subsection{ Setting up the induction scheme} Fix $\lambda \in 2^{\N}$. We begin with few basic definitions.
\begin{definition}
Let $1 \leq L \leq d$. Define $\modelcurves{B,L}$ to be the collection of all curves $\gamma : I \rightarrow \R^d$ such that for all $s \in I$, we have
\begin{align}\label{eq: class of curves}
            \norm{\gamma}_{C^{2d}(I)} \leq B \qquad \textrm{ and} \qquad \big| \det\begin{pmatrix}
            \gamma^{(1)}(s) & \cdots &  \gamma^{(L)}(s)
            \end{pmatrix} \big|\geq {B}^{-1},
\end{align}
where the square of the determinant is interpreted as the sum of squares of its $L \times L$ minors.
\end{definition}
\begin{definition} \label{def: symb_type_L}
 Let $1 \leq L \leq d$ 
and $\gamma \in \modelcurves{B,L}$. A symbol $a \in C^{3d}( \R^d \times I \times I)$ is said to be of type $(\lambda, A,L)$ with respect to $\gamma$ if the following hold:
    \begin{enumerate}[i)]
        \item There exists a constant $C = C(A,B) > 1$, independent of $\lambda$, such that 
        \[ \spsupp{\xi}{a} \subseteq  \{ \xi \in \R^d : C\lambda \leq |\xi| \leq 2C\lambda \}.\]
        \item $|\partial_{s}^{\beta}a(\xi,s,t)| \lesssim_{\beta,A} 1 $ for $ 0 \leq \beta \leq 3d$ and $(\xi,s,t) \in \supp{a}$.
            \item The inner product estimate 
     \begin{gather}
        A^{-1}|\xi| \leq  \sum_{i = 1}^L |\langle \gamma^{(i)}(s),\xi \rangle| \leq A |\xi| \qquad \textrm{ holds for all $(\xi,s) \in \spsupp{\xi,s}{a}$.} \label{eq: nondegassum}
\end{gather}
 \end{enumerate}
\end{definition}
Proposition~\ref{prop: main} is consequence of the following result.
\begin{proposition}\label{prop: main_L}
 Fix $1 \leq L \leq d$, $\gamma \in \modelcurves{B,L}$ and let $a$ be a symbol of type $(\lambda,A,L)$ with respect to $\gamma$. Then,
\begin{align*}
\norm{\fravg{a}{\gamma}}_{L^2(\R^{d+1}) \rightarrow L^2(\R^{d+1})} \lesssim_{A,B,d} (\log \lambda)^{(L-1)/2}.
 \end{align*}
 \end{proposition}
 In view of \eqref{eq: nondegassum}, it is clear that Proposition~\ref{prop: main} corresponds to the case $L = d$ of Proposition~\ref{prop: main_L}. 
 
 Proposition~\ref{prop: main_L} is proved by inducting on $L$. Given an arbitrary symbol $a \in C^{3d}(\R^d \times I \times I)$ and a smooth curve $\gamma$, we present here a general argument which will be used repeatedly through the induction process in order to obtain favourable norm bounds for the Fourier integral operator $\fravg{a}{\gamma}$. For $g \in \schw{d}$, we aim for the estimate
\begin{align} \label{eq: uniform_op_bound}
    \norm{\fravg{a}{\gamma}g}_{L^2(\R^{d+1})} \lesssim_{A,B,d} \norm{g}_{L^2(\R^{d+1})}.
\end{align}
By applying Plancherel's theorem and the Cauchy--Schwarz inequality,
\begin{align*}
    \norm{\fravg{a}{\gamma}g}_{L^2(\R^{d+1})}^2 
    &=\int_{\R} (1+ |\sigma|) |\ft{x,s}{\avg{a}{\gamma}g}|^2(\sigma,\xi) \dxi \mathrm{d}\sigma \\
    & \lesssim \norm{\avg{a}{\gamma}g}_{L^2(\R^{d+1})}\norm{(1+ \sqrt{-\partial_{s}^{2}})\avg{a}{\gamma}g}_{L^2(\R^{d+1})} \\
    & \leq \norm{\avg{a}{\gamma}g}_{L^2(\R^{d+1})}^2 \\
    &  \qquad +\norm{\sqrt{-\partial_{s}^{2}}\avg{a}{\gamma}g}_{L^2(\R^{d+1})} \norm{\avg{a}{\gamma}g}_{L^2(\R^{d+1})}.
    \end{align*}
Since the Hilbert transform is bounded on $L^2$,
\[ \norm{\sqrt{-\partial_{s}^{2}}\avg{a}{\gamma}g}_{L^2(\R^{d+1})} \lesssim \norm{\partial_{s}\avg{a}{\gamma}g}_{L^2(\R^{d+1})}.\]
Thus, to prove \eqref{eq: uniform_op_bound}, it suffices to show that there exists $\Lambda > 1$ such that
\[\norm{\partial_{s}^{\iota}\avg{a}{\gamma}}_{L^2(\R^{d+1}) \rightarrow L^2(\R^{d+1})} \lesssim_{A,B,d} \Lambda^{(2\iota - 1)/2} \quad \textrm{for } \iota = 0,1. \]
Applying Plancherel's theorem and the Cauchy--Schwarz inequality, 
 \begin{align*}
     \norm{\avg{a}{\gamma}g}_{L^2(\R^{d+1})}^2  & \sim_{d}  \int_{I}\int_{\R^d} \mathcal{B}[a]\ft{x}{g}(\xi,t) \overline{\ft{x}{g}(\xi,t)} \dxi \dt \\
     & \leq  \int_{\R^d}\norm{\mathcal{B}[a]\ft{x}{g} (\xi,\cdot)}_{L^2(\R)} \norm{\ft{x}{g}(\xi,\cdot)}_{L^2(\R)} \dxi,
\end{align*}
where $\mathcal{B}[a]$ is the operator that integrates (in $t'$ variable) functions against the kernel 
\begin{align}\label{eq: kernel expression}
    K[a](\xi,t',t) := \int_{I} e^{i \inn{ (t -t')\gamma(s)}{\xi}} a(\xi,s,t') \overline{a(\xi,s,t)} \ds.
 \end{align}
At this point, note that $\partial_{s}\avg{a}{\gamma}g$ can be expressed as $\avg{\mathfrak{d}_sa}{\gamma}g$, with a symbol
     \[ \mathfrak{d}_sa(\xi,s,t) := t\inn{\gamma'(s)}{\xi} a(\xi,s,t) + \partial_{s}a(\xi,s,t) \quad \textrm{ for $(\xi,s,t) \in \R^{d+2}.$}\]
Applying Schur's test, we see that \eqref{eq: uniform_op_bound} is a consequence of the estimates
 \begin{equation}\label{eq: desired_ker_estimates}
     \sup_{(\xi,t') \in \mathbb{A}_{\lambda} \times I} \norm{K[\mathfrak{d}_s^{\iota}a](\xi,t',\cdot)}_{L^1_t(I)} \lesssim_{A,B,d} \Lambda^{2\iota - 1} \quad \textrm{ for } \iota = 0,1, 
 \end{equation}
completing the discussion.

The first application of this reduction is the following lemma.
\begin{lemma}[Base case] \label{lem: base case}
 Proposition~\ref{prop: main_L} holds when $L = 1$.
 \end{lemma}
 \begin{proof}
Choose a curve $\gamma$ and a symbol $a$ that satisfies the assumptions of Proposition~\ref{prop: main_L} with $L = 1$. In particular, $a$ is of type $(\lambda, A,1)$ with respect to $\gamma$ and as a consequence,
\begin{equation*}
    |\inn{\gamma'(s)}{\xi}| \sim_{A} \lambda \qquad \textrm{holds for 
 } (\xi,s) \in \spsupp{\xi,s}{a}.
\end{equation*}
Following the previous discussion, we wish to obtain good decay estimates for the function $K[\mathfrak{d}_s^{\iota}a]$ with $\iota = 0,1$. Integrating-by-parts in \eqref{eq: kernel expression} and using Definition~\ref{def: symb_type_L}~ii), we have
\[|K[\mathfrak{d}_s^{\iota}a](\xi,t',t)| \lesssim_{A,B,N} \lambda^{2\iota}(1 + |t - t'|\lambda)^{-N} \qquad \textrm{for } \iota = 0,1 \textrm{ and } N \geq 1.\]
Clearly, these decay estimates imply the required bounds \eqref{eq: desired_ker_estimates} with $\Lambda = \lambda$. Consequently, we obtain \eqref{eq: uniform_op_bound} with the implicit constant depending only on $A$, $B$ and $d$.
\end{proof}
Lemma~\ref{lem: base case} addresses the base case of Proposition~\ref{prop: main_L}. It remains to establish the inductive step.
 \begin{proposition}\label{prop: induction}
Suppose the statement of Proposition~\ref{prop: main_L} is true for $L = N-1$. Then it is also true for $L = N$.
 \end{proposition}
 Proposition~\ref{prop: main_L}, and therefore Theorem~\ref{thm: main}, follow from Proposition~\ref{prop: induction} and Lemma~\ref{lem: base case}. For the remainder of the section we present the proof of Proposition~\ref{prop: induction}, which is broken into steps.
 \subsection{Initial decomposition}\label{subsec:initial decomp}
To begin the proof of Proposition~\ref{prop: induction}, let $\gamma$ and $a$ be chosen to satisfy the assumptions of the Proposition~\ref{prop: main_L} with $L = N$. We apply a natural division of the symbol $a$. Let 
$H : \R^{d+1} \rightarrow \R$ be defined as the product \begin{align*}
H(\xi,s) := \prod_{i = 1}^{N-1} \eta(A' \lambda^{-1} \inn{\gamma^{(i)}(s)}{\xi} )
\end{align*} 
where $A'$ is large constant which will be chosen depending only on $A$, $B$ and $N$. Here $\eta$ is as defined in \eqref{eq: littlewood_paley}. Note that
\[|\partial_{s}^{\beta}H(\xi,s)| \lesssim_{\beta,A,B} 1 \qquad \textrm{ for $(\xi,s) \in \spsupp{\xi,s}{a}$ and $\beta \in \N \cup \{0\}$}.\]
Furthermore, \eqref{eq: nondegassum} holds for the pair $(\gamma, a(1-H))$ with $A$ replaced with $A'$ and $L = N-1$. Thus, $a(1-H)$ is a symbol of type $(\lambda,A',N - 1)$ with respect to $\gamma$. Applying the induction hypothesis, we deduce the desired estimate for the part of the operator corresponding to the symbol $a(1-H)$.

Since \eqref{eq: nondegassum} holds with $L = N$ in $\supp{a}$ by assumption, the inequalities 
\begin{gather}
        (10A)^{-1}|\xi| \leq |\inn{\gamma^{(N)}(s)}{\xi}|  \leq A|\xi|,  \label{eq: large Lth inn prod} \\
   \sum_{i = 1}^{N-1} |\langle \gamma^{(i)}(s),\xi \rangle|  \leq 10^{-10}A^{-1}|\xi| \label{eq: L-1 degeneracy} 
\end{gather}
also hold for all $(\xi,s) \in \spsupp{\xi,s}{aH}$, provided $A'$ is chosen large enough depending on $N$ and $A$. Henceforth, for simplicity, we write $a$ in place of $aH$ and therefore work with the stronger assumptions \eqref{eq: large Lth inn prod} and \eqref{eq: L-1 degeneracy} on the support of $a$. An application of the implicit function theorem now shows that for any $\xi \in \spsupp{\xi}{a}$, there exists $\sigma(\xi) \in I$ with
\begin{align}\label{eq: existence of sigma}
     \inn{\gamma^{(N-1)}\circ \sigma(\xi)}{\xi} = 0.
\end{align}

The strategy now involves a decomposition the symbol away from the most degenerate regions in $\R^{d+1}$. Set 
\[ G(\xi,s) := \sum_{i = 1}^{N-1} |\ezero^{-1}\lambda^{-1}\inn{\gamma^{(i)}\circ \sigma(\xi)}{\xi}|^{{\cexp}/{(N-i)}} + \ezero^{-2}|s - \sigma(\xi)|^{\cexp},\] where the constant $\ezero = \ezero(A,B)$ will be chosen small enough to satisfy the forthcoming requirements of the proof. The function $G$ should be interpreted as the function measuring the distance of $(\xi,s)$ from the co-dimension $N$ surface 
\[\Gamma := \{(\xi,s) \in \R^d \times I: \inn{\gamma^{(i)}\circ \sigma(\xi)}{\xi} = 0 \text{ for } 1 \leq i \leq N-1 \text{ and } |s - \sigma(\xi)| = 0 \}.\]
We now decompose the $(\xi,s)$-space dyadically away from $\Gamma$. Suppose $\eta_1, \beta_1 \in C_{c}^{\infty}(\R)$ are chosen such that 
\begin{align}\label{eq: littlewood_paley2}
\supp{\eta_1} \subseteq \{ r \in \R: |r| \leq 4 \}, \qquad \supp{\beta_1} \subseteq \{ r \in \R: {1}/{4} \leq |r| \leq 4 \}    
\end{align}
and 
\begin{align*}
    \eta_1(r) + \sum_{n \in \N} \beta_1(2^{-2n}r) = 1 \qquad \textrm{for all $ r \in \R$.}
\end{align*} Set 
\begin{align}\label{eq: def a_n}
 a^{n}({\xi,s,t}) := a({\xi,s,t}) \cdot \begin{cases} \eta_1( \eone^{2}\lambda^{2/N} G(\xi,s)) & \textrm{if } n = 0, \\
 \beta_1(\eone^{2}2^{-2n}\lambda^{2/N} G(\xi,s)) &\textrm{if } n \ge 1.
 \end{cases}
\end{align}
where $\eone$ will be chosen small enough (depending on $\ezero$) to satisfy the forthcoming requirements of the proof. Observe that $a = a^0 + \sum_{n \in \N} a^n$ and this automatically induces a similar decomposition for the Fourier integral operator $\fravg{a}{\gamma}$. 

Since 
\begin{align}\label{eq: uniform_bd_G}
    |G(\xi,s)| = O_{B,d}(\ezero^{-2}) \qquad \textrm{ for all $(\xi,s) \in \spsupp{\xi,s}{a}$},
\end{align}
the symbols $a^n$ are trivially zero except for $O_{A,B}(\log \lambda)$ many values of $n$. Thus, by Plancherel's theorem, 
\begin{align}
    \norm{\fravg{\sum_{n \in \Z}a^n}{\gamma}g}_{L^2(\R^{d+1})}^2 
   & =  \sum_{n \in \Z} \norm{\fravg{ a^{n}}{\gamma}g}_{L^2(\R^{d+1})}^2 \notag \\
    & \lesssim_{A,B} |\log \lambda| \max_{n \in \Z} \norm{\fravg{ a^{n}}{\gamma}g}^2_{L^2(\R^{d+1})}. \label{eq: reducing to a_n}
\end{align}

In light of the above, it remains to bound the fractional operator ${\fravg{a^n}{\gamma}}$ for different values of $n$. The case of $n = 0$ is dealt with by the following lemma. 
\begin{lemma}\label{lem: n = 0 case}
    \begin{align}\label{eq: n = 0 case}
         \norm{\fravg{ a^{0}}{\gamma}}_{L^2(\R^{d+1}) \rightarrow L^2(\R^{d+1})}  \lesssim_{A,B,d} 1.
    \end{align}
\end{lemma}
Next lemma addresses the case of all other values of $ n$.
\begin{lemma}\label{lem: n > 0} For any $n \geq 1$, we have 
\begin{align} \label{eq: n > 0}
     \norm{\fravg{a^{n}}{\gamma}}_{L^2(\R^{d+1}) \rightarrow L^2(\R^{d+1})} 
& \lesssim_{A,B,d} (\log \lambda)^{{(N-2)}/{2}}.
\end{align}
\end{lemma}
Assuming Lemma~\ref{lem: n = 0 case} and Lemma~\ref{lem: n > 0} for now, we  plug \eqref{eq: n = 0 case}, \eqref{eq: n > 0} into \eqref{eq: reducing to a_n} and obtain
\[\norm{\fravg{a}{\gamma}g}_{L^2(\R^{d+1})} \lesssim_{A,B,d} (\log\lambda)^{(N-1)/2}\norm{g}_{L^2(\R^{d+1})}.\]
This concludes the proof of Proposition~\ref{prop: induction}.

Rest of the section is dedicated to the proofs of the two key lemmas (Lemma~\ref{lem: n = 0 case} and Lemma~\ref{lem: n > 0}).
\subsection{Proof of Lemma~\ref{lem: n = 0 case}} 
 To prove Lemma~\ref{lem: n = 0 case}, we do not appeal to the induction hypothesis but directly estimate the operator.
\begin{proof}[Proof of Lemma~\ref{lem: n = 0 case}]
In view of discussions around \eqref{eq: uniform_op_bound} and \eqref{eq: desired_ker_estimates}, it suffices to show 
\begin{equation}\label{eq: ker_est_Nth_step}
    |K[\mathfrak{d}_{s}^{\iota}a^{0}](\xi,t',t)| \lesssim_{A,B,d} \lambda^{(2\iota - 1)/{N}} \qquad \textrm{ for  $\iota = 0,1$ and $(\xi,t',t) \in \R^d \times I \times I$}.
\end{equation}
Indeed, \eqref{eq: ker_est_Nth_step} implies \eqref{eq: desired_ker_estimates} with $a = a^{0}$ and $\Lambda = \lambda^{1/N}$, which in turn gives \eqref{eq: n = 0 case} as discussed above.

The estimate \eqref{eq: ker_est_Nth_step} for $\iota = 0$ is immediate from \eqref{eq: kernel expression} as the $\spsupp{s}{a^{0}(\xi,\cdot,\cdot)}$ is contained in an interval of length $O_{A,B}(\lambda^{-1/N})$ for any fixed $\xi \in \R^d$. By \eqref{eq: kernel expression} again, the case $\iota = 1$ becomes evident once we verify the estimates 
\begin{align*}
    |\inn{\gamma'(s)}{\xi}| + |\partial_{s}(a^{0})(\xi,s,t)| \lesssim_{A,B,d} \lambda^{{1}/{N}} \qquad \text{ for $(\xi,s,t) \in \supp{a^{0}}$.}
\end{align*}
It is easy to see that $|\partial_{s} (a^{0})(\xi,s,t)| \lesssim_{A,B} \lambda^{1/N}$.  To estimate the remaining term, note that for any $1 \leq i \leq N$, we have
\begin{align*}
|\inn{\gamma^{(i)}\circ \sigma(\xi)}{\xi}| \lesssim_{A,B,N} \lambda \lambda^{(i- N)/N} \qquad \text{and} \qquad |s - \sigma(\xi)|\lesssim_{A,B} \lambda^{-1/N}  
\end{align*}
for $(\xi,s) \in \spsupp{\xi,s}{a^{0}}$. Using Taylor's theorem,
\begin{align*}
    |\inn{\gamma^{(1)}(s)}{\xi}| & \leq \sum_{j = 1}^{N-1} |\inn{\gamma^{(j)}\circ \sigma(\xi)}{\xi}| \frac{|s - \sigma(\xi)|^{j - 1}}{(j - 1)!} + B|\xi|\frac{|s- \sigma(\xi)|^{N - 1}}{(N-1)!} \\
    & \lesssim_{A,B,d} \lambda^{1/N}
\end{align*}
for $(\xi,s) \in \spsupp{\xi,s}{a^{0}}$, as required. Thus, we obtain \eqref{eq: ker_est_Nth_step} and consequently \eqref{eq: n = 0 case}.
\end{proof}
\subsection{Further decomposition}\label{subsec: further decomp}
In order to prove Lemma~\ref{lem: n > 0} we must introduce a further decomposition of the symbol. Let $\zeta \in C_{c}^{\infty}(\R)$ be chosen such that $\supp{\zeta} \subseteq~[-1,1]$ and $\sum_{\nu \in \Z}  \zeta(\,\cdot\, - \nu) = 1$. For $n \in \mathbb{N}$ and $\nu \in \Z$, consider the symbol
\begin{align}\label{eq: defining local_symb}
 a^{n,\nu}({\xi,s,t}) := a^{n}({\xi,s,t})\zeta(2^{-n}\lambda^{1/N}(s - s_{n,\nu}))
\end{align}
where $s_{n,\nu} := 2^n\lambda^{-{1}/{N}}\nu$. Observe that the original symbol is recovered as the sum
\begin{align}\label{eq: further_decomp}
     a = \sum_{n= 0}^{C\log(\lambda)} a^{n} = a^{0} + \sum_{n= 1}^{C\log(\lambda)}\sum_{\nu \in \Z} a^{n,\nu} ,
\end{align}
where $C$ is a constant that depends only on $A,B$. 
The following lemma records a basic property of the localised symbols, which is useful later in the proof. 
\begin{lemma}\label{lem: inn prod estimaTes}
Let $n \geq 1, \nu \in \Z$ and $\rho := 2^n\lambda^{-{1}/{N}}$.
 For any $(\xi,s) \in \spsupp{\xi,s}{a^{n,\nu}}$, we have
        \begin{align}\label{eq: unscaled_sum}
            \sum_{i = 1}^{N-1} \rho^{i - N}|\inn{\gamma^{(i)}(s)}{\xi}|  \sim_{A,B,d} |\xi| \sim \lambda.
        \end{align}

\end{lemma}

\begin{proof}
The upper bound in \eqref{eq: unscaled_sum} is easier to prove than the lower bound and follows from a similar argument. Consequently, we will focus only on the lower bound.

Fix $n \geq 1$ and $\nu \in \Z$. Recall from the definitions that 
\begin{align}\label{eq: g_localisation}
     \eone^{-2}/4 \leq \sum_{i = 1}^{N-1} |\ezero^{-1}\lambda^{-1}\rho^{i - N}\inn{\gamma^{(i)}\circ \sigma(\xi)}{\xi}|^{2/(N-i)} + |\ezero^{-1}\rho^{-1}(s - \sigma(\xi))|^2 \leq 4\eone^{-2}
\end{align}
for all $(\xi,s) \in \spsupp{\xi,s}{a^{n,\nu}}$. Fixing $\xi$, we now consider two cases depending on which terms of the above sum dominate.\medskip

\noindent \textit{Case 1}. Suppose $(\ezero \eone^{-1}\rho)/4 \leq |s - \sigma(\xi)|$.
By Taylor's theorem, there exists $s_{\ast} \in I$ between $ s$ and $\sigma(\xi)$ 
such that
\[ \inn{\gamma^{(N-1)}(s)}{\xi} -  \inn{\gamma^{(N-1)}\circ\sigma(\xi)}{\xi} = \inn{\gamma^{(N)}(s_{\ast})}{\xi}(s-\sigma(\xi)).\]
Combining this with \eqref{eq: large Lth inn prod} and \eqref{eq: existence of sigma}, we deduce that $|\inn{\gamma^{(N-1)}(s)}{\xi}| \gtrsim_{A} \lambda (\ezero \eone^{-1}\rho)$.  This gives the lower bound in \eqref{eq: unscaled_sum}.
\medskip

\noindent \textit{Case 2}.
Suppose Case 1 fails. Using \eqref{eq: g_localisation}, we can find $1 \leq i_0 \leq N-2$ such that
\begin{equation}\label{eq: eq1}
    c_N \ezero\lambda (\eone^{-1}\rho)^{N - i_0}\leq |\inn{\gamma^{(i_0)}\circ \sigma(\xi)}{\xi}| 
  \leq 2^N\ezero\lambda (\eone^{-1}\rho)^{N - i_0},
 \end{equation}
with $c_N := (4N)^{-N}$,  whilst  $|s - \sigma(\xi)| \leq \ezero \eone^{-1}\rho$ and
\begin{equation*}
    |\inn{\gamma^{(i)}\circ \sigma(\xi)}{\xi}|  \leq 2^N\ezero\lambda (\eone^{-1}\rho)^{N - i} \qquad \text{ for all $i_0 < i \leq N-1$.}
\end{equation*}
By Taylor's theorem,
\begin{align}
    |\inn{\gamma^{(i_0)}(s)}{\xi}- & \inn{\gamma^{(i_0)} \circ \sigma(\xi)}{\xi}| \notag \\
     &\leq \sum_{i = i_0 + 1}^{N - 1} 2^N\ezero^{1+i - i_0}\lambda (\eone^{-1}\rho)^{N - i} (\eone^{-1}\rho)^{i - i_0} + B\lambda (\ezero \eone^{-1}\rho)^{N-i_0} \notag \\ 
     & \leq (c_N \ezero/2) \lambda (\eone^{-1}\rho)^{N - i_0}, \label{eq: eq2}
\end{align}
provided the constant $\ezero$ is chosen small enough depending on $B$ and $N$. Combining \eqref{eq: eq1} and \eqref{eq: eq2}, we deduce that
\[ |\inn{\gamma^{(i_0)}(s)}{\xi}| \sim_{\epsilon_0} \lambda (\eone^{-1}\rho)^{N - i_0} \qquad \textrm{for all $s \in \spsupp{s}{a^{n,\nu}}$,}\]
which implies the lower bound in \eqref{eq: unscaled_sum}.
\end{proof}
In view of \eqref{eq: further_decomp}, we restrict our attention to $\fravg{a^{n,\nu}}{\gamma}$ for fixed $n \in \N$ and $\nu \in \Z$. Before proceeding to its analysis, we make the following elementary observation about the size of $\rho := 2^{n}\lambda^{-1/N}$. From the definition \eqref{eq: littlewood_paley2} of $\beta_1$, note that
 \[\eone^{-2}/4 \leq \rho^{-2}G(\xi,s) \quad \textrm{for} \quad (\xi,s) \in \spsupp{\xi,s}{a^{n}}.\] Combining this with \eqref{eq: uniform_bd_G}, we deduce that 
$\rho = O_{B,d}( \eone \ezero^{-1}).$ Thus, by choosing $\eone$ small enough (depending on $\ezero,B,d$), we can assume that
\begin{align}\label{eq: small_rho}
    \rho \leq B^{-2d}.
\end{align}

In the following subsections, the norm bounds for the operator $\fravg{a^{n,\nu}}{\gamma}$ are obtained using the induction hypothesis via a method of rescaling. 
\subsection{Rescaling for the curve}\label{subsec: rescaling} 
In this subsection, we introduce the rescaling map in a generic setup and describe its basic properties which will play a crucial role in the proof of Lemma~\ref{lem: n > 0}. 

For $\gamma \in \modelcurves{B,N}$ and $s_{\circ} \in I$, let \[V_{s_0}^N := \mathrm{span}\{\gamma^{(1)}(s_0),\dots, \gamma^{(N)}(s_0)\}.\]
Using \eqref{eq: class of curves}, note that $\dim V_{s_0}^N = N$. For $0 < \rho < 1,$ define a linear operator $T_{s_{\circ}, \rho}^N$ such that
\begin{align}\label{eq: defining T_n1}
    T_{s_{\circ}, \rho}^N \left(\gamma^{(i)}(s_0)\right) :=        \rho^{i} \gamma^{(i)}(s_0)  \qquad \textrm{for $1 \leq i \leq N$}
\end{align}
and
\begin{align*}
T_{s_{\circ}, \rho}^N v = \rho^{N} v \qquad \textrm{ for $v \in (V_{s_0}^N)^{\perp}$}.
\end{align*} 
It is clear that $T_{s_{\circ}, \rho}^N$ is a well-defined map such that 
\begin{align}\label{norm_of_rescaling_map}
    \norm{(T_{s_{\circ}, \rho}^N)^{-1}} \lesssim_{B} \rho^{-N}.
\end{align}
Supposing $[s_{\circ} - \rho, s_{\circ} + \rho] \subseteq I$, we define the rescaled curve
\begin{equation*}
    \gamma_{s_{0},\rho}^N(s)  := \big(T_{s_{\circ}, \rho}^N\big)^{-1} (\gamma(s_{\circ} +\rho s) - \gamma(s_{\circ})).
\end{equation*}
For simplicity, we introduce the notation 
\begin{align}\label{eq: rescaled_notations} 
T := T^{N}_{s_{0},\rho}, \quad  T^* := \big(T^{N}_{s_{0},\rho}\big)^{-\top}\quad \textrm{and} \quad \tilde{\gamma} := \gamma_{s_{0},\rho}^N.
\end{align}

The following lemma verifies nondegeneracy assumptions for the rescaled curve. 
\begin{lemma}\label{lem: curve rescaled to type N-1}
For $0 < \rho \leq B^{-2d}$ and $\gamma \in \modelcurves{B,N}$, the rescaled curve $\tilde{\gamma}$ as in \eqref{eq: rescaled_notations} lies in $\modelcurves{B_1,N-1}$ where $B_1$ depends only on $B$ and $N$. 
\end{lemma}
A key feature of Lemma~\ref{lem: curve rescaled to type N-1} is that the parameter $B_1$ is independent of $\rho$.
\begin{proof}[Proof of Lemma~\ref{lem: curve rescaled to type N-1}]
We begin by verifying the first part of \eqref{eq: class of curves} for the curve $\widetilde{\gamma}$. From the definition, we see that $\tilde{\gamma}^{(i)}(s) = \rho^{i} T^{-1} \gamma^{(i)}(s_{0} + \rho s)$ for any $i \in \N$. Combining this identity with \eqref{eq: class of curves} and \eqref{norm_of_rescaling_map}, we deduce that
\begin{align}\label{eq: estimates large i}
\norm{\tilde{\gamma}^{(i)}}_{L^{\infty}(I)} = O_{B}( \rho) \qquad \textrm{whenever $N+1 \leq i \leq 2d$.}    
\end{align} 
Let $1 \leq i \leq N$. By Taylor's theorem, \eqref{eq: defining T_n1} and \eqref{norm_of_rescaling_map}, we have
\begin{align}
    \tilde{\gamma}^{(i)}(s)
    & = \rho^{i} \sum_{j = i}^{N} T^{-1} \gamma^{(j)}(s_{0})\frac{(\rho s)^{j-i}}{(j-i)!} + O_{B}(\norm{T^{-1}}\rho^{N+1}) 
 \notag \\
    & = \sum_{j = i}^{N} \gamma^{(j)}(s_{0})\frac{s^{j-i}}{(j-i)!} + O_{B}(\rho) \label{eq: Taylor_exp_gamma}
\end{align}
Combining \eqref{eq: Taylor_exp_gamma} with \eqref{eq: class of curves}, we obtain uniform size estimates for $\tilde{\gamma}^{(i)}(s)$ when $1 \leq i \leq N$. Together with \eqref{eq: estimates large i}, this implies 
\begin{align}\label{eq: C_d_norm_gamma}
    \norm{\tilde{\gamma}}_{C^{2d}(I)} \lesssim_{B} 1.
\end{align} 

It remains to verify the second part in \eqref{eq: class of curves} for the curve $\tilde{\gamma}$ and $L = N-1$. In view of \eqref{eq: C_d_norm_gamma}, it suffices to obtain a lower bound for the determinant of the $d \times N$ matrix whose columns vectors are formed by $(\tilde{\gamma}^{(i)})_{1\leq i \leq N}$ . Observe that using the multilinearity of the determinant and elementary column operations, \eqref{eq: Taylor_exp_gamma} gives 
\begin{align*}
    |\det\begin{pmatrix}
        \tilde{\gamma}^{(1)}(s) & \cdots & \tilde{\gamma}^{(N)}(s) 
    \end{pmatrix}| = |\det\begin{pmatrix}
        {\gamma}^{(1)}(s_{n,\nu}) & \cdots & {\gamma}^{(N)}(s_{n, \nu}) 
    \end{pmatrix}|  + O_B(\rho).
\end{align*}
By the hypothesis of the lemma, $\rho$ is small enough so that the above identity combined with \eqref{eq: class of curves} gives the estimate
\begin{align*}
    |\det\begin{pmatrix}
        \tilde{\gamma}^{(1)}(s) & \cdots & \tilde{\gamma}^{(N)}(s) 
    \end{pmatrix}| \geq (2B)^{-1}.
\end{align*}
Now, an application of \eqref{eq: C_d_norm_gamma} (in particular, $|\tilde{\gamma}^{(N)}(s)| \lesssim_{B} 1$) completes the proof of \eqref{eq: class of curves} for $\gamma = \tilde{\gamma}$, $L = N-1$ and $B$ replaced with a new constant $B_1$.
\end{proof}
The rescaling map $T^{N}_{s_0,\rho}$ can be used to introduce a rescaling for the operators we are interested in. This is done in the next subsection.
\subsection{Rescaling for the operator}\label{subsec: rescaling op}
To introduce the operator rescaling, we begin by considering a Schwartz function $u: \R \rightarrow \R$. Let $s_0 \in I$ and $0 < \rho < 1$. Direct computations gives
\begin{align*}
    [(1+ \sqrt{-\partial_{s}^{2}})^{1/2}u](s_0 + \rho s) = \rho^{-1/2}[(\rho + \sqrt{-\partial_{s}^{2}})^{1/2} \tilde{u}](s), 
\end{align*}
where $\tilde{u}(s) := u(s_0 + \rho s)$. Thus,
\begin{align}
    \norm{(1+ \sqrt{-\partial_{s}^{2}})^{1/2}u}_{L^2(\R)}^2
    & \sim  \int_{\R}|(\rho + |\sigma|)^{1/2}\ft{s}{\tilde{u}}(\sigma)|^{2} \mathrm{d}\sigma  \notag \\ 
    & \leq \int_{\R}|(1 + |\sigma|)^{1/2}\ft{s}{\tilde{u}}(\sigma)|^{2} \mathrm{d}\sigma \notag\\ 
    & = \norm{(1+ \sqrt{-\partial_{s}^{2}})^{1/2} \tilde{u}}_{L^2(\R)}^{2}. 
\label{eq: eq3}
\end{align}

For an arbitrary symbol $a \in C^{2d}(\R^d \times I \times I)$ and $\gamma \in \modelcurves{B,N}$, recall the definition of $\avg{a}{\gamma}$ from \eqref{eq: defining fio avg}. Temporarily fixing $x \in \R^d$, set 
\begin{align}\label{eq: rescaling_step1}
    u(s) = \avg{a}{\gamma}g (x,s) \qquad \textrm{and} \qquad \tavg{a}{\gamma}g(x,s) := \avg{a}{\gamma}g(x,s_0 + \rho s).
\end{align}  By combining \eqref{eq: eq3} for each $x \in \R^d$ with Fubini's theorem, 
\begin{align}\label{eq: s-rescaling}
    \norm{\fravg{a}{\gamma}g}_{L^2(\R^{d+1})} \lesssim \norm{(1+ \sqrt{-\partial_{s}^{2}})^{1/2}\tavg{a}{\gamma}g}_{L^2(\R^{d+1})}.
\end{align}

 We claim that for $(x,s) \in \R^{d+1}$, the identity
\begin{align}\label{eq: rescaled operator}
   \tavg{a}{\gamma}g(x,s) =  |\det{T^*}|^{1/2}\avg{\tilde{a}}{\widetilde{\gamma}}\tilde{g}(T^{-1}x, s)
\end{align}
holds with $T$, $\tilde{\gamma}$ as in \eqref{eq: rescaled_notations}, symbol
\begin{align*}
    \tilde{a}(\xi,s,t) := a(T^*\xi,t,s_{0} +\rho s)
\end{align*}
and input function $\tilde{g}$ defined by
\begin{align*}
     \ft{x}{\tilde{g}}(\xi,t) :=  |\det{T^*}|^{1/2} e^{it \inn{T^{-1}\gamma(s_{0})}{\xi}}\ft{x}{g}(T^*\xi,t).
\end{align*}

Verifying \eqref{eq: rescaled operator} is just a matter of unwinding the definitions. First, we expand $\tavg{a}{\gamma}g(x,s)$ using \eqref{eq: rescaling_step1} as the oscillatory integral
\begin{align*}
 \int_{\R^{d} \times I} e^{i \inn{x - t(\gamma(s_{0} +\rho s) - \gamma(s_{0}))}{\xi}} a(\xi,s_{0} +\rho s,t) e^{i t\inn{\gamma(s_{0})}{\xi}} \ft{x}{g}(\xi,t) \dxi \dt.
\end{align*}
Applying change of variables $\xi \rightarrow T^* \xi$, the above expression can be written as 
\begin{align}
   |\det{T^*}|^{1/2} \int_{\R^d \times \R} e^{i \inn{(T^{-1}x - t\widetilde{\gamma}(s)}{\xi}}  &a(T^*\xi,s_{0} +\rho s,t)  \ft{x}{\tilde{g}}(T^*\xi,t) \dxi \dt \label{eq: oscill_exp}\\
    & = |\det{T^*}|^{1/2}(\avg{\tilde{a}}{\widetilde{\gamma}}\tilde{g})(T^{-1}x, s), \notag
\end{align}
proving the claim \eqref{eq: rescaled operator}. \medskip

Fix $n \in \N$, $\nu \in \Z$ and recall the definitions of $a^{n,\nu}$ and $s_{n,\nu}$ from \S\ref{subsec: further decomp}. Consider the rescaling map $T$ as defined in \S\ref{subsec: rescaling} for \[s_0 = s_{n,\nu} \qquad \textrm{and} \qquad \rho = 2^n \lambda^{-1/N}.\] Furthermore, consider the operator rescaling as in \eqref{eq: rescaled operator} for $a = a^{n,\nu}$.
In this setup, we record some of the basic properties of how $T^*$ (as defined in \eqref{eq: rescaled_notations}) interacts with $\tilde{a}$.
\begin{lemma}\label{lem: rescaling properties} 
The rescaling map $T^*$ satisfies the estimate
\begin{equation}\label{eq:T acts on resc symb}
     \rho^{-N}|\xi| \lesssim_{A,B} |T^* \,\xi| \lesssim_{B}\rho^{-N} |\xi| \qquad \text{for all} \quad \xi \in \spsupp{\xi}{\tilde{a}}.
\end{equation}
\end{lemma}

\begin{proof}
Fix $1 \leq i \leq N$. From the definition of $T$, we have
    \begin{align}\label{eq: eq5}
        \inn{\gamma^{(i)}(s_{n,\nu})}{\xi} = \rho^{i}\inn{T^{-1} \gamma^{(i)}(s_{n,\nu})}{\xi} = \rho^{i}\inn{\gamma^{(i)}(s_{n,\nu})}{T^* \xi}.
    \end{align}
Fix  $\xi \in \spsupp{\xi}{\tilde{a}}$ so that, by the definition of the rescaled symbol,  $T^* \, \xi \in \spsupp{\xi} a^{n,\nu}$. Using Lemma~\ref{lem: inn prod estimaTes} when $i \leq N-1$ and the Cauchy--Schwarz inequality (or \eqref{eq: large Lth inn prod}) when $i = N$, we obtain
\[\big|\inn{\gamma^{(i)}(s_{n,\nu})}{T^* \xi}\big| \lesssim_{A,B} \rho^{N-i} \big|T^* \xi\big| \quad \textrm{for $1 \leq i \leq N$}.\] Combining this with \eqref{eq: eq5}, we deduce that 
\begin{equation}\label{eq: eq6}
    |\inn{\gamma^{(i)}(s_{n,\nu})}{\xi}| \lesssim_{A,B} \rho^{N}|T^* \, \xi|.
\end{equation}
On the other hand, if $ v \in (V^{N}_{s_{n,\nu}})^{\perp}$ is a unit vector, one can argue as in \eqref{eq: eq5} to have
\begin{equation}\label{eq: eq7}
    |\inn{v}{\xi}| = |\rho^N\inn{v}{T^* \xi}| \leq |T^* \xi| 
\end{equation}
where the fact $\rho < 1$ has been used.
Combining \eqref{eq: eq6}, \eqref{eq: eq7} and \eqref{eq: class of curves}, we obtain the lower bound in \eqref{eq:T acts on resc symb}. The upper bound follows from \eqref{norm_of_rescaling_map}.
\end{proof}
The following lemma now verifies how rescaling improves the type condition of the symbol. 
\begin{lemma}\label{lem: rescaled to type N-1}
 The rescaled symbol $\tilde{a}$ is of type $(\rho^{N}\lambda, A_1, N-1)$ with respect to $\widetilde{\gamma}$, where both $A_1$ depends only on $A, B$ and $N$. 
\end{lemma}

\begin{proof}
By Lemma~\ref{lem: rescaling properties}, it is clear that \[ \spsupp{\xi}{\tilde{a}} \subseteq \{\xi \in \R^d: |\xi| \sim_{A,B}  \rho^{N}\lambda \}.\] 
Since $0 < \rho < 1$ by \eqref{eq: small_rho}, the estimates $|\partial_{s}^{\beta}\tilde{a}(\xi,s,t)| \lesssim_{\beta} 1$ for $(\xi,s,t) \in \supp{\tilde{a}}$ follows from similar derivative estimates for $a$. Thus, it remains to verify that \eqref{eq: nondegassum} holds for the rescaled setup for $L = N-1$, $\gamma = \widetilde{\gamma}$ and $a = \tilde{a}$. Explicitly, we wish to show 
\[
\sum_{i = 1}^{N-1} |\inn{\widetilde{\gamma}^{(i)}(s)}{\xi}| \sim_{A,B} |\xi| \quad \textrm{for all } (\xi,s) \in \spsupp{\xi,s}{\tilde{a}}.\]
To this end, we recall from Lemma~\ref{lem: inn prod estimaTes} that
\begin{align}\label{eq: unscaled N-1-deg}
    \sum_{i = 1}^{N-1} \rho^{i - N}|\inn{\gamma^{(i)}(s)}{\xi}| \sim_{A,B} \lambda \quad \textrm{ for all }(\xi,s) \in \spsupp{\xi,s}{a^{n,\nu}}.
\end{align}    
However, by unwinding the definition,
\begin{align*}
     \langle\widetilde{\gamma}^{(i)}(s),\xi \rangle = 
      \rho^{i} \langle T^{-1}\gamma^{(i)}(s_{n,\nu} + \rho s),\xi \rangle  = \rho^{i} \langle \gamma^{(i)}(s_{n,\nu} + \rho s), T^*\,\xi \rangle.
\end{align*}
Thus, by \eqref{eq: unscaled N-1-deg} and Lemma~\ref{lem: rescaling properties}, 
\begin{align*}
    \sum_{i = 1}^{N-1} |\langle\widetilde{\gamma}^{(i)}(s),\xi \rangle| \sim_{A,B} \rho^{N}|T^*\,\xi| \sim |\xi| \quad \textrm{for all } (\xi,s) \in \spsupp{\xi,s}{\tilde{a}},
\end{align*}
which is the required estimate.
\end{proof} 
\subsection{Proof of Lemma~\ref{lem: n > 0}}
With all the available tools, the operator estimate for $\fravg{a^{n}}{\gamma}$ for $n \geq 1$ now follows easily.
\begin{proof}[Proof of Lemma~\ref{lem: n > 0}]
Fix $n \geq 1$. Temporarily fix $\nu \in \Z$. In view of \eqref{eq: s-rescaling} and \eqref{eq: rescaled operator} for $a = a^{n,\nu}$, we have 
\begin{align}\label{eq: summing_up_rescaling}
    \norm{\fravg{a^{n,\nu}}{\gamma}g}_{L^2(\R^{d+1})} \lesssim \norm{\fravg{\tilde{a}}{\widetilde{\gamma}}\tilde{g}}_{L^2(\R^{d+1})}.
\end{align}
Suppose $\tilde{\zeta} \in C_c^{\infty}(\R)$ is chosen such that $\supp{\tilde{\zeta}} \subseteq [-4,4]$, $\tilde{\zeta}(r) = 1$ when $|r| \leq 3$ and 
\[\sum_{\nu \in \Z} \tilde{\zeta}(\, \cdot \, - \nu) \lesssim 1. \]
In view of the support properties of $a^{n,\nu}$ (in particular, \eqref{eq: def a_n} and \eqref{eq: defining local_symb}), we have 
\[\tilde{\zeta}(\ezero^{-1}\eone\rho^{-1}(\sigma(T^{\ast}\xi) - s_{n,\nu})) = 1 \qquad \textrm{for $\xi \in \spsupp{\xi}{\tilde{a}}$.}\]
Consequently, recalling the integral expression \eqref{eq: oscill_exp}, it is clear that one can replace $\tilde{g}$ with $\tilde{g}^{n,\nu}$ in \eqref{eq: summing_up_rescaling} where 
\[\tilde{g}^{n,\nu} := \tilde{\zeta}\left(\ezero^{-1}\eone\rho^{-1}(\sigma \circ T^{\ast}(\tfrac{1}{i}\partial_{x}) - s_{n,\nu})\right) \tilde{g}. \]
Now, Lemma~\ref{lem: curve rescaled to type N-1} and Lemma~\ref{lem: rescaled to type N-1} ensure that the rescaled pair ($\tilde{a}, \tilde{\gamma}$)  satisfy the assumptions of Proposition~\ref{prop: main_L} with $L = N-1$ (note that \eqref{eq: small_rho} ensures that $\rho$ is of the right size, so Lemma~\ref{lem: curve rescaled to type N-1} applies). Thus, the statement of the proposition applies and we obtain
\begin{align}
    \norm{\fravg{\tilde{a}}{\widetilde{\gamma}}\tilde{g}^{n,\nu}}_{L^2(\R^{d+1})} &
    \lesssim_{A,B,d} (\log \rho^{N}\lambda)^{{(N-2)}/{2}} \norm{\tilde{g}^{n,\nu}}_{L^2(\R^{d+1})} \notag\\
    & \lesssim_{A,B,d} (\log \lambda)^{{(N-2)}/{2}} \norm{\tilde{g}^{n,\nu}}_{L^2(\R^{d+1})}. \label{eq: norm_est_fixed_n,nu}
\end{align}
Thus, the proof of Lemma~\ref{lem: n > 0} reduces to summing the above estimates in $\nu$ without further loss in $\lambda$. Using \eqref{eq: further_decomp}, Plancherel's theorem and the support properties of symbols $a^{n,\nu}$, we combine \eqref{eq: norm_est_fixed_n,nu} for different values of $\nu$ to deduce that
\begin{align*}
     \norm{\fravg{a^{n}}{\gamma}g}_{L^2(\R^{d+1})}^{2} & \lesssim_{d}  \sum_{\nu \in \Z} \norm{\fravg{a^{n,\nu}}{\gamma}g}_{L^2(\R^{d+1})} ^2 \notag \\
     &  \lesssim_{A,B,d} (\log \lambda)^{{N-2}} \sum_{\nu \in \Z} \norm{\tilde{g}^{n,\nu}}_{L^2(\R^{d+1})}^2.
\end{align*}
After a change of variable, it is evident that 
$\norm{\tilde{g}^{n,\nu}}_{L^2(\R^{d+1})} =  \norm{g^{n,\nu}}_{L^2(\R^{d+1})}$, where
\[g^{n,\nu} := \tilde{\zeta}\left(\ezero^{-1}\eone\rho^{-1}(\sigma (\tfrac{1}{i}\partial_{x}) - s_{n,\nu})\right) g. \]
Thus, by another application of the Plancharel's theorem,
\begin{align*} 
     \norm{\fravg{a^{n}}{\gamma}g}_{L^2(\R^{d+1})}^{2} & \lesssim_{A,B,d} (\log \lambda)^{{N-2}}
\sum_{\nu \in \Z}\norm{g^{n,\nu}}_{L^2(\R^{d+1})}^2  \\
& \lesssim_{A,B,d} (\log \lambda)^{{N-2}} \norm{g}_{L^2(\R^{d+1})}^2
\end{align*}
concluding the proof.
 
\end{proof}
In the next section, we discuss the sharpness of the main theorem.
\section{Sharpness of the Theorem~\ref{thm: main}}\label{sec: sharpness of main thm}
By acting the maximal function on standard test functions, here we discuss the sharpness of Theorem~\ref{thm: main} in two directions: sharpness of the range of $p$ and the dependence of the operator norm on $\log \delta^{-1}$.
\subsection{Sharpness of the range of \textit{p}}
    Fix $p \in [1,\infty)$ and assume that given any $\epsilon > 0$, we have
    \begin{equation}\label{eq: exampleeq1}
        \norm{\maxfn}_{L^p(\R^{d+1}) \rightarrow L^p(\R^{d})} \lesssim_{\epsilon} \delta^{-\epsilon} \qquad \textrm{for all $0 < \delta < 1$.}
    \end{equation} 
    Temporarily fix $\epsilon$ and $\delta$. Set $g_{\delta} := \chi_{B(0,\delta)}$. It is easy to show that the $\delta$-neighbourhood of the curve $-\gamma$ is a subset of the super-level set 
    \[\{ x \in \R^d : |\maxfn g_{\delta}(x)| \gtrsim \delta \}.\]
    Applying Chebyshev's inequality and using \eqref{eq: exampleeq1}, we have
    \[ \delta \delta^{(d-1)/p} \lesssim_{\epsilon} \delta^{(d+1)/p - \epsilon}.\]
    Letting $\delta \rightarrow 0$, we see that $p \geq 2 - \epsilon$. Letting $\epsilon \rightarrow 0$, we conclude that $p \geq 2$. Thus, $L^p$ operator norm of $\maxfn$ has polynomial blowup in $\delta^{-1}$ for $p \in [1,2)$.

\subsection{ Sharpness of the operator norm} Fix $0 < \delta < 1$ and $p \in [2,\infty)$. Consider the vectors $\boldsymbol{w} = (x,0), \ \boldsymbol{z} = (y,0)  \in \R^{d+1}$. It follows from the definition that 
    \[ \boldsymbol{w} + T_{\delta}(r) \cap \boldsymbol{z} + T_{\delta}(s) \neq \emptyset \] if and only if  there exist a $t \in [-1,1]$ such that
    \begin{align}\label{eq: tube_intersection}
        (x - y) + t(\gamma(r) - \gamma(s)) = O(\delta).
    \end{align}
Assuming $|\gamma(s)| \sim 1$ for all $s \in [-1,1]$, it is also not hard to see that 
\begin{align}\label{eq: tube_vol_estimates}
     \mathrm{Vol}_{\R^{d+1}}(\boldsymbol{w} + T_{10\delta}(r) \cap \boldsymbol{z} + T_{\delta}(s)) \sim \frac{\delta^{d+1}}{\delta + |\gamma(r) - \gamma(s)|}
\end{align}
whenever \eqref{eq: tube_intersection} holds.

    Fixing $(x,r) \in \R^d \times I$, set $f_{\delta} := \chi_{ \boldsymbol{w} + T_{10\delta}(r)}$ and note that $\norm{f_{\delta}}_{L^p(\R^{d+1})} \sim \delta^{d/p}$. Fix $0 \leq k \leq \lfloor \log(\delta^{-1}) \rfloor$, define 
\[A_{k} := \{ y \in \R^d : |\maxfn f_{\delta}(y)| \sim 2^{-k} \}.\]
We claim that 
\[|A_k| \gtrsim 2^{2k}\delta^{d}.\]
    Indeed, in view of  \eqref{eq: tube_vol_estimates}, $A_{k}$ contains all points $y \in \R^d$ for which there exist $s,t \in  [-1,1]$ such that \eqref{eq: tube_intersection} holds and $|\gamma(s) - \gamma(r)| \sim 2^{k}\delta$. The latter condition ensures that the admissible directions $\gamma(s)$ belong to a portion of the curve which is contained inside a ball of radius $ \sim 2^{k}\delta$. Moreover, for a fixed direction $\gamma(s)$, any $y \in \R^d$ that lies in the $\delta$-neighbourhood of the tube $x + \{ t(\gamma(r) - \gamma(s)): t \in [-1,1] \}$ satisfies \eqref{eq: tube_intersection}. Therefore, $A_k$ contains the $\delta$-neighbourhood of a two-dimensional cone in $\R^d$ of diameter $ \sim 2^k\delta$, justifying our claim. Thus,
\begin{align*}
    (\log \delta^{-1}) \delta^{d} \lesssim     \sum_{k = 0}^{\lfloor \log(\delta^{-1}) \rfloor} 2^{-2k}|A_k| \leq \norm{\maxfn}_{L^p(\R^{d+1}) \rightarrow L^p(\R^{d})}^p \norm{f_{\delta}}^p_{L^p(\R^{d+1})}.
\end{align*}
Consequently, we see that \[\norm{ \maxfn }_{L^p(\R^{d+1}) \rightarrow L^p(\R^{d})} \gtrsim (\log \delta^{-1})^{1/p}.\]
In view of the above, we may conjecture that $(\log \delta^{-1})^{{1}/{p}}$ is the sharp $L^p$ operator norm of $\maxfn$ for $p \in [2,\infty)$.
In other words, it is possible that Theorem~\ref{thm: main} gives only a partial result in this direction. \\

In the next section, we discuss a generalisation of Theorem~\ref{thm: main}.
\section{Further extensions}\label{sec: further ext}
As observed in \cite{bghs-lp}, a stronger version of Theorem~\ref{thm: main} which deals with \textit{families of anisotropic tubes} is used in actual applications to the proofs of certain geometric maximal estimates (such as that of the helical maximal function). In this section, we state the anisotropic extension of Theorem~\ref{thm: main} with a brief discussion of how the argument presented in the article can be modified to work for the more general setup.

We begin by introducing the anisotropic tubes using the Frenet frame co-ordinate system. For $s \in I$, let $\{e_1(s),\dots,e_d(s)\}$ denote the collection of Frenet frame basis vectors, formed by applying Gram--Schmidt process to the set $\{\gamma^{(1)}(s), \dots, \gamma^{(d)}(s)\}$. For $\vecrr = (r_1,\dots,r_d) \in (0,1)^{d}$, we consider a tube in $\R^{d+1}$ in the direction of $\gamma(s)$, defined as 
\begin{align}\label{eq: def_anisotropic_tubes}
    {T}_{\mathbf{r}}(s) := \big\{(y,t) \in \R^{d} \times I : |\inn{y - t\gamma(s)}{e_j(s)}| \leq r_j \text{ for } \ 1 \leq j \leq d \big\}.
\end{align}
As before, we can introduce the corresponding averaging and maximal operator as 
\begin{align*}
    \mathcal{A}_{\vecrr}^{\gamma} g(x,s) & : = \frac{1}{|T_{\vecrr}(s)|}\int_{{T}_{\vecrr}(s)} g(x - y,t) \dy \dt \qquad \text{ for } (x,s) \in \R^d \times I
\end{align*}
and 
\begin{align*}
 \mathcal{N}_{\vecrr}^{\gamma}g(x) & : = \sup_{s \in I} |\mathcal{A}_{\vecrr}^{\gamma}g(x,s)| \qquad \text{ for } x \in \R^d
\end{align*}
whenever $g \in L^1_{\mathrm{loc}}(\R^{d+1})$. 

By modifying the argument presented in \S\ref{sec: initial reduct} and \S\ref{sec: proof of Main prop}, the $L^p$ boundedness problem for $\mathcal{N}_{\vecrr}^{\gamma}$ can be resolved under mild hypothesis on $\vecrr$ . Our result \cite{ags2023} is as follows.
\begin{theorem}\label{thm: aniso}
Let $\vecrr = (r_1,\dots,r_d) \in (0,1)^{d}$ be chosen such that
\begin{equation}\label{eq: admissible_r}
      r_{d} \leq \cdots \leq r_{1} \leq r_2^{1/2} \qquad \textrm{and} \qquad r_{j} \leq r_i^{{(k-j)}/{(k-i)}} r_k^{{(j-i)}/{(k-i)}} \quad 
 \end{equation}
for any $1 \leq i \leq j \leq k \leq d$ hold. Then, there exists $C_{d,\gamma} > 1$ such that 
    \begin{align*} 
\norm{\mathcal{N}_{\vecrr}^{\gamma}}_{L^2(\R^{d+1}) \rightarrow L^2(\R^{d})} \leq C_{d,\gamma} (\log r_d^{-1})^{d/2}.
    \end{align*}
\end{theorem}
There are two most interesting cases where Theorem~\ref{thm: aniso} can be applied. These are when $\vecrr = \vecrr_{\rt{iso}} := (\delta,\dots, \delta)$ and $\vecrr = (\delta, \delta^{2},\dots, \delta^{d})$ for $0 < \delta < 1$. In both cases, it is clear that $\vecrr$ satisfies \eqref{eq: admissible_r}. By applying Theorem~\ref{thm: aniso} in the first case, we recover Theorem~\ref{thm: main} as a consequence. 
\begin{proof}[Discussion on the proof of Theorem~\ref{thm: aniso}]
In the following discussion, we only emphasize the major changes from the arguments presented in this article. A detailed proof can be found in \cite{ags2023}.

We begin by recalling the definition 
\begin{align*}
 \ar (\xi,s,t) := \psi(\delta|\xi|) \tilde{\chi}_{I}(s)\tilde{\chi}_{I}(t) \qquad \textrm{ for $(\xi,s,t) \in \R^{d} \times I \times I$}
\end{align*}
from \eqref{eq: symb_defn}. In view of \eqref{eq: def_anisotropic_tubes}, the symbol equivalent of $a_{\vecr}$ has a different form in the anisotropic setup. We define
\begin{align*}
 \arr (\xi,s,t) := \prod_{j = 1}^{d} \psi( \inn{\xi}{e_j(s)} r_j) \tilde{\chi}_{I}(s)\tilde{\chi}_{I}(t) \qquad \textrm{ for $(\xi,s,t) \in \R^{d} \times I \times I$.}
\end{align*}
Note that when $\vecrr = \vecrr_{\rt{iso}}$, we essentially recover $\ar$ from this definition.

By arguing as in \S\ref{sec: initial reduct}, we can reduce the proof of Theorem~\ref{thm: aniso} to establishing operator norm estimates for the Fourier integral operator $\fravg{\arr}{\gamma}$. In particular, it suffices to show that 
\begin{align}\label{eq: aniso_fio_est}
    \norm{\fravg{\arr}{\gamma}}_{L^2(\R^{d+1}) \rightarrow L^2(\R^{d+1})} \lesssim_{d,\gamma} (\log r_d^{-1})^{d-1}.
\end{align}
In \S\ref{sec: initial reduct}, we obtained an equivalent version of \eqref{eq: aniso_fio_est} for $\vecrr = \vecrr_{\rt{iso}}$ by first dyadically decomposing the operator and then applying Proposition~\ref{prop: main} to each part. The proposition, in turn, was proved using an induction argument (in particular, Proposition~\ref{prop: main_L}). Similarly, we can reduce the proof of \eqref{eq: aniso_fio_est} to a modified form of Proposition~\ref{prop: main_L}. The modifications lie in the analysis of the derivative estimates for the underlying symbol. The core argument, involving the decomposition as described in \S\ref{subsec:initial decomp}, \S\ref{subsec: further decomp} and the rescaling as described in \S\ref{subsec: rescaling}, \S\ref{subsec: rescaling op} remain intact. We will now briefly outline what the modifications are and the motivations behind them.

Recall that the derivative bounds
\begin{align}\label{eq: derivative estimates discussion}
  |\partial_{s}^{\beta}\ar (\xi,s,t)| \lesssim_{\beta,\gamma,d} 1 \qquad \textrm{ for $\beta \in \N$ and $(\xi,s,t) \in \supp{\ar}$}
\end{align}
were explicitly used for directly estimating parts of the operator at many stages in the proof of Proposition~\ref{prop: main_L} (in particular, see the proofs of Lemma~\ref{lem: base case} and Lemma~\ref{lem: n = 0 case}). 
Consequently, the operator norm of $\fravg{a_{\vecr}}{\gamma}$ depend on the upper bound in \eqref{eq: derivative estimates discussion}.
Note that the rescaling did not interfere with \eqref{eq: derivative estimates discussion} and therefore, one was able to carry these estimates unchanged throughout the induction process (see Definition~\ref{def: symb_type_L}).

The situation is significantly different in the anisotropic setup. In contrast to \eqref{eq: derivative estimates discussion}, the best attainable $L^{\infty}$ bounds for the derivatives of the anisotropic symbol are
\begin{align*}
  \norm{\partial_{s}^{\beta}\arr}_{L^{\infty}(\R^d \times I \times I)} \lesssim_{\beta,\gamma,d} \max_{1 \leq j_1,\dots,j_\beta,k_1,\dots, k_{\beta} \leq d} \ \prod_{i = 1}^{\beta} r_{j_i} r_{k_i}^{-1} \quad \textrm{ for $\beta \in \N$.}
\end{align*}
Note that the expression on the right depends on $\vecrr$ and can be extremely large. However, after applying the decomposition as described in \S\ref{subsec:initial decomp} and \S\ref{subsec: further decomp}, it is possible to obtain improved $L^{\infty}$ bounds for the derivatives of each part of the symbol. This suggests that in contrast to Definition~\ref{def: symb_type_L} ii), the induction assumption in the anisotropic setup should include pointwise bounds for the derivatives of the symbol \textit{expressed in a form that is sensitive to the many decomposition in the argument}.

Finer control over the derivatives of the symbol at each step of the induction as mentioned above, is insufficient on its own for the purpose of establishing acceptable bounds at stages where we directly estimate the operator. The assumptions \eqref{eq: admissible_r} are the minimal conditions required such that, coupled with the modified induction assumptions and additional properties of the decomposition, they can be used to satisfactorily estimate the derivatives of the symbol at these stages.

By introducing the above-mentioned changes in the assumptions about the underlying symbol and taking care of many additional technical details (in particular, we must also keep track of the constants we gain during each application of rescaling), we can prove a modified version of Proposition~\ref{prop: main_L}, completing the proof of Theorem~\ref{thm: aniso}.
\end{proof}

\bibliography{bibmine}{}
\bibliographystyle{plain}

\end{document}